\documentclass[12pt, oneside]{article}   	
\usepackage[margin = 1.0 in]{geometry}
\usepackage{graphicx}				

\usepackage{mathrsfs}								
\usepackage{amssymb}
\usepackage{amsmath}
\usepackage{amsthm}
\usepackage{mathtools}
\usepackage{setspace}
\usepackage{hyperref}

\usepackage{fancyhdr}
\newtheorem{thm}{Theorem}
\newtheorem*{exmp*}{Example}

\newtheorem{prop}[thm]{Proposition}
\newtheorem{lem}[thm]{Lemma}
\newtheorem{conj}[thm]{Conjecture}

\newtheorem{defn}[thm]{Definition}

\theoremstyle{remark}
\newtheorem{rem}{Remark}

\usepackage{mathtools}
\newcommand{\defeq}{\vcentcolon=}

\DeclareMathOperator{\rk}{rank}
\DeclareMathOperator{\Lie}{Lie}
\DeclareMathOperator{\Log}{Log}
\DeclareMathOperator{\Jac}{Jac_C}
\newcommand\abint[2]{\prescript{\text{AB}}{}{\int_{#1}^{#2}}}
\newcommand\bcint[2]{\prescript{\text{BC}}{}{\int_{#1}^{#2}}}

\DeclareMathOperator{\red}{red}

\newcommand{\ts}{\textsuperscript}

\title{Rank-Favorable Bounds for Rational Points on Superelliptic Curves of Small Rank}
\author{Noam Kantor}

\begin{document}

\maketitle

\begin{abstract}
Let $C$ be a curve of genus at least three defined over a number field, and let $r$ be the rank of the rational points of its Jacobian. Under mild hypotheses on $r$, recent results by Katz, Rabinoff, Zureick-Brown, and Stoll bound the number of rational points on $C$ by a constant that depends only on its genus. Yet one expects an even stronger bound that depends favorably on $r$: when $r$ is small, there should be fewer points on $C$. In a 2013 paper, Stoll established such a ``rank-favorable" bound for hyperelliptic curves using Chabauty's method. In the present work we extend Stoll's results to superelliptic curves, noting in the process some differences that ought to inform uniformity conjectures for general curves. Our results have stark implications for bounding numbers of rational points, since $r$ is expected to be small for ``most" curves.
\end{abstract}





\section{Background and Main Theorems}

Diophantine finiteness, in its most general sense, is the philosophy that the set of rational points on a variety of general type should be ``small," for example not Zariski dense. There is a web of interconnected conjectures that make this philosophy precise, such as the Bombieri--Lang conjecture (see \cite[p.~479]{hindrysilverman}). In the dimension one case, Bombieri--Lang is just Faltings' Theorem, which says that the number of rational points on a curve of genus greater than one is finite.

But it is computationally and theoretically important to have control of rational points beyond the basic finiteness supplied by Faltings. In fact, one conjectures the following.

\begin{conj}[\cite{uniformity}, p.~2]
Let $C$ be a curve of genus $g>1$ defined over a number field $K$. Then there exists a constant $B(g,K)$ such that $\# C (K) < B(g,K)$.
\end{conj}

Significant progress has been made on this conjecture when the rank of the Jacobian of $C$ is somewhat smaller than $g$ using a technique called Chabauty's method. The classical application of the method constructs a $p$-adic differential $\omega$ and an associated $p$-adic analytic function $f_\omega$ on $\Jac(\mathbb{C}_p)$ such that the zero set of $f_\omega$ contains $C(\mathbb{Q})$. Chabauty's original paper reasons using basic facts about analytic functions to show that $f_\omega$ only has finitely many zeros. In fact, $f_\omega$ is really a $p$-adic integral of $\omega$, and a robust theory of $p$-adic integration is the key to modern applications of Chabauty's ideas.

Indeed, a number of improvements have been made to Chabauty's method which make it both stronger and more widely applicable. First, Coleman \cite{coleman1985} realized that one could apply Riemann--Roch and Newton polygons to explicitly estimate the number of zeros of $\omega$, and thus the number of zeros of $f_\omega$. Another limitation of the original method was that it required analysis in $\mathbb{C}_p$, for $p$ a prime of good reduction for $C$. Since the first prime of good reduction might be arbitrarily large, even Coleman's Riemann--Roch improvement gave non-uniform bounds on the number of rational points on $C$. Michael Stoll, in the paper \cite{stoll2013uniform}, described a theory of $p$-adic integration on annuli that allowed for uniform bounds using primes of bad reduction. (Lorenzini--Tucker \cite{lt} had actually already used primes of bad reduction, but they required a regular model of the curve in question, which led to non-uniform bounds.)  We mention also the recent results of Ellenberg and Hast \cite{EH}, which reproves Faltings' theorem for superelliptic curves using non-abelian Chabauty, and those of Katz--Rabinoff--Zureick-Brown \cite{krzbsurvey}, which use Stoll's ideas and tropical geometry to prove uniform bounds for general curves satisfying a rank hypothesis.

Prior to Stoll's work, it was known that the $p$-adic integral that Chabauty and Coleman used in the case of good reduction is actually two different integrals that happen to agree for a prime of good reduction. These two integrals, the Berkovich--Coleman and Abelian integrals, no longer agree in the bad reduction case; for example, the ``tube" $$] y [ \, = \{x \in C(\mathbb{C}_p) \mid \red(x) = y\}$$ is an annulus when $y$ is a nodal point the special fiber. (Here $\red$ refers to the reduction map from $C(\mathbb{C}_p)$ to $C(\overline{\mathbb{F}}_p)$.)

Stoll realized that one can use differentials for which the two integrals agree, and forcing the two integrals to be equal in this way consists of one linear condition on the space of differentials. (He imposes one more linear condition to kill the $p$-adic logarithm used in the Berkovich--Coleman integral.) When $C$ is a hyperelliptic curve, Stoll explicitly describes the differentials on each tube. He then applies Newton polygon arguments and a little linear algebra on each tube, and then incorporates a sophisticated analysis of the combinatorics of the special fiber to prove the following.

\begin{thm}[\cite{stoll2013uniform}, Theorem 1.4]
Suppose $C \colon y^2 = f(x)$ is a hyperelliptic curve defined over $\mathbb{Q}$. Assume further that $g \geq 3$ and $r \defeq \rk\Jac(\mathbb{Q}) \leq g-3$. Then
\begin{align*}
\# C(\mathbb{Q}) &\leq 33(g-1) + 1 &\text{ if } r=0 \text{, and } \\ \# C(\mathbb{Q}) &\leq 8rg + 33(g-1) - 1 &\text{ if } r \geq 1.
\end{align*}
\end{thm}

Note here that the given bound is linear in both $g$ and $r$. Katz, Rabinoff, and Zureick-Brown \cite{katz2015uniform} have shown that the number of rational points on \textit{any} curve defined over $\mathbb{Q}$ with $r < g-2$ is bounded quadratically by $g$. Based on Stoll's results, however, one actually expects a bound that is linear in both $r$ and $g$ and reduces to their bound when $r$ is close to $g$.

In this paper we extend Stoll's work to some superelliptic curves, that is, curves of the form $y^m = f(x)$ for $m \geq 2$ and $f$ a rational polynomial of degree at least $4$, with roots of multiplicity strictly less than $m$. We obtain a linear bound in $g$ and $r$ for each fixed $m$, thus confirming the idea that small values of $r$ should lead to better Chabauty bounds.

The following theorem is our main result.
\begin{thm} \label{main}

Let $C$ be the superelliptic curve defined above. Suppose $m>2$ and $r \leq \lfloor \frac{\deg(f)}{m} \rfloor - 4$. Let $p$ be the smallest prime congruent to 1 modulo $m$. Then 
$$\# C(\mathbb{Q}) < (8g-8)(r+3) + 2m(r+3) + (2p + 2)(g-1) + 4r.$$

\end{thm}

Much of Stoll's work transfers from the hyperelliptic case to the superelliptic case, but a number of interesting differences arise when $m > 2$ that shed light on uniformity conjectures for general curves. Perhaps surprisingly, the most technical input for our generalization is Raynaud's classification of order $m$ automorphisms of the $p$-adic disc and annulus when $(p,m) = 1$. (Incidentally, the $p \mid m$ case is much more involved and plays a large part in the solution of the local lifting conjecture in \cite{oort}.)



\section*{Acknowledgements}
This research would never have been possible without the help of my advisor David Zureick-Brown and my graduate student advisor Jackson Morrow. I am grateful also to Andrew Obus, who directed me to Raynaud's work on quotients of semistable curves. Last but not least, I am indebted to the entire Emory Mathematics Department and the Robert W. Woodruff Scholarship for supporting me throughout my mathematical journey thus far.

\subsection{Notation}

We now fix some notation for the rest of our discussion. From here onward, $p$ will always be a prime that does not divide $m$. We will let $K$ be a number field, and $k$ the completion of $K$ at some prime $\mathfrak{p}$ lying over $p$. We denote by $\zeta_m$ a fixed primitive $m$\ts{th} root of unity in $k$ or $K$ depending on the context. For any extension of $p$-adic fields, $e$ will denote the ramification index of the extension, and $f$ will denote the residue degree.

Let $C$ be a superelliptic curve over a number field $K$ defined by an affine equation $$y^m = f(x) = \prod_{j=1}^s (x - \theta_j)^{n_j}$$ for some $m > 2$ and $f \in K[X]$.  We will assume $C$ defines a curve of genus at least three, and that $n_j < m$ for each $j$. In any case, $C$ comes with an automorphism $\tau$ of order $m$ given by $y \mapsto \zeta_m y$, and the quotient of $C$ by the subgroup generated by $\tau$ is $\mathbb{P}^1$. We will let $\pi$ denote the projection from $C$ to $\mathbb{P}^1$.

Recall that the $p$-adic unit open disc of radius $r$ is just the set $$D_{0,k} = \{y \in \mathbb{C}_p \mid |y| < 1\}.$$ Now suppose $\alpha$ is in the value group of $k$. We define the $p$-adic annulus with outer radius one and inner radius $\alpha$ by $$A_{\alpha,k} = \{y \in \mathbb{C}_p \mid \alpha < |y| < 1\}.$$ We have emphasized the field $k$ in this notation because it is important that all changes of coordinate respect base fields. Finally, $C$ has a $g$-dimensional vector space of $p$-adic differentials denoted $H^0(C_k,\Omega^1_C)$.

\section{Chabauty's Method and Stoll's Improvement}

\subsection{The Technique}

We now describe in more detail the technology behind Chabauty's method. For this section $C$ is any curve of genus at least two. Now \cite[Section 5]{stoll2013uniform} describes a model of $C$ for which every rational point is contained in a $p$-adic disc or annulus. This model mimics the ideal situation in which $C$ is semistable, since in that case if $x$ is any point of the special fiber of $C$, then $]x[$ is either a disc or an annulus depending on whether $x$ is smooth or singular, respectively (see \cite[Proposition 2.3]{bl}). From here on we suppose that there is some $K$-rational $P \in \, ]x[$\,; otherwise our bounds are trivially satisfied.

The key lemma of Chabauty's method revolves around a $p$-adic integral which vanishes on the $p$-adic closure of $\Jac(\mathbb{Q})$.

\begin{defn}

Let $A$ be an abelian variety over $\mathbb{C}_p$. The abelian logarithm on $A$ is the unique homomorphism of $\mathbb{C}_p$-Lie groups $\log \colon A(\mathbb{C}_p) \to \Lie(A)$ such that $$d\log \colon \Lie(A) \to \Lie(\Lie(A)) = \Lie(A)$$ is the identity map. Now by definition, $\Lie(A)$ is the dual of $H^0(A,\Omega^1_A),$ and we denote the evaluation pairing by $\langle \cdot,\cdot \rangle$.

Finally, for $x \in A(\mathbb{C}_p)$ and $\omega \in H^0(A,\Omega^1_A)$ we set $$\abint{O}{x} \omega \defeq \langle \log(x), \omega \rangle $$ and we call $\abint{}{}$ the abelian integral on $A$. We also define, for $y \in A(\mathbb{C}_p)$, $\abint{x}{y} \omega = \abint{O}{x} \omega - \abint{O}{y} \omega$. (We have let $O$ denote the identity of $A$.) When $C$ is a curve and $\iota$ is an Abel--Jacobi mapping into its Jacobian, we define the integral from $x \in C$ to $y \in C$ by integrating from $\iota(x)$ to $\iota(y)$ in $\Jac$.

\end{defn}

Fixing a residue tube $]x[$\,, Chabauty's method centers around the subspace $V_{\text{chab}}$ of differentials such that $$f_\omega(z) \defeq \abint{P}{z} \omega = 0$$ for all $z \in \Jac(\mathbb{Q}) \cap \, ]x[\,$. Since $C(\mathbb{Q}) \subseteq \Jac(\mathbb{Q})$, every rational point of $C$ is a zero of $f_\omega$ for all $\omega \in V_{\text{chab}}$. The following lemma powers Chabauty's method.

\begin{lem} \label{vchab}
Let $r \defeq \rk(\Jac(\mathbb{Q}))$. Then $\dim(V_{\text{\emph{chab}}}) \geq g-r$.
\end{lem}

The second integral that one can define on the curve $C$ is called the Berkovich--Coleman integral -- it is essentially defined by formal anti-differentiation and is therefore (1) dependent just on the endpoints of the integral on annuli, discs, and more general ``wide opens" and (2) it is amenable to the tools of Newton polygons.

\begin{defn}
Suppose $\omega$ is a differential on the annulus $A_{\alpha,k}$ with local coordinate $T$. Consider the local power series expansion of $\omega$: $$\omega = g(T) \frac{dT}{T} = \sum_{n=-\infty}^{\infty} a_n T^n \frac{dT}{T}$$ where $g(T)$ converges on $A_{\alpha,k}$. Denote by $f$ the (incomplete) formal antiderivative of $g$ given by $$\sum_{n \neq 0} \frac{a_{n-1}}{n} T^{n}.$$
Finally, let $$\Log(T) = \sum_{n=1}^\infty (-1)^{n+1} \frac{(T-1)^n}{n} $$ be the branch of the logarithm where $\Log(p) = 0$. Then we define the Berkovich--Coleman integral on $C$ via the formula $$\bcint{x}{y} \omega = (f(y) + a_0 \Log(y)) - (f(x) - a_0 \Log(x)).$$

\end{defn}

Stoll proves in \cite{stoll2013uniform} the following integral comparison theorem, which \cite{katz2015uniform} reconceptualizes using analytic uniformization of $p$-adic abelian varieties.

\begin{thm} \label{stoll}
Suppose $x,y \in C(k)$ both lie in the same annulus $\phi \colon A_{\alpha, k} \to C$. Then there is a codimension two subspace of $H^0(C_k,\Omega^1_C)$ such that for every differential $\omega$ in that subspace, $$\abint{x}{y} \omega = \bcint{x}{y} \omega.$$ 
\end{thm}

Now, by Lemma~\ref{vchab} and Theorem~\ref{stoll}, there exists a differential which satisfies the equality of the Berkovich--Coleman and Abelian integrals and also lies in $V_{\text{chab}}$: if $r < g-2$ we conclude that there is a nonzero subspace $W_{\text{chab}}$ of differentials such that $\dim(W_{\text{chab}}) \geq g-r-2$ and the Abelian and Coleman integrals are equal on $W_{\text{chab}}$. Since the integrals agree on $W_{\text{chab}}$, they both vanish on the $p$-adic closure of the rational points of the Jacobian. The main work of this paper is to bound the number of zeros of such a differential on a residue tube.

To summarize: there is a ``relatively large" subspace $W_{\text{chab}}$ such that: (1) the integrals of elements of $W_{\text{chab}}$ vanish on $C(\mathbb{Q})$ (owing to the Abelian integral), and (2) the integral can be computed using formal antiderivatives (owing to the Berkovich--Coleman integral).

\section{Automorphisms of an Annulus}
Emulating \cite{stoll2013uniform}, we proceed with the following process. A rational point on the curve $C$ is contained in a disc or an annulus, so it is sent to another disc or annulus by the superelliptic automorphism (it may be sent to itself). Since we know the possible automorphisms of a disc and annulus, we know the behavior of the discs and annulus under the quotient map. Finally, once we know the ramification behavior of the quotient map we use it to write explicit equations for the annulus on the curve. (By this, we mean explicit equations of the map from the standard annulus to the curve.) Having this equation in hand, we can explicitly describe a basis for the differentials on the annulus or disc.

In \cite{stoll2013uniform}, Stoll classifies the involutions of the $p$-adic disc and annulus ``by hand." We here describe more general results of Raynaud that classify finite order automorphisms of the $p$-adic annulus and disc using algebraic methods. The following theorem is a corollary of \cite[Propositions 2.3.1 and 2.3.2]{raynaud1999}. Recall that the ring of analytic functions on the annulus $A_{\alpha, k}$ is $\mathcal{O}[[X,Y]]/(XY = p^e)$ for some integer $e \geq 1$. Its reduction thus consists of two branches, which may be either fixed or interchanged by an automorphism of the annulus.

\begin{thm}\label{raynaud}
\begin{enumerate}
\item[]
\item
Let $\tau \colon D_{0,k} \to D_{0,k}$ be an analytic map of order $m$, with $m$ coprime to $p$. Then after an analytic change of coordinates defined over $k$, $\tau$ is just multiplication by an $m$\ts{th} root of unity. In particular, $\tau$ has only one fixed point and $D_{0,k}/ \langle \tau \rangle$ is a disc.

\item
Let $\tau \colon A_{\alpha,k} \to A_{\alpha,k}$ be an analytic map of order $m$ such that $\tau$ fixes each of the branches of the reduction of $A_{\alpha,k}$. Then after an analytic change of coordinates defined over $k$, $\tau$ is just multiplication by an $m$\ts{th} root of unity. In particular, $\tau$ has no fixed points and $A_{\alpha,k}/ \langle \tau \rangle$ is an annulus.

\item
Now suppose $m$ is even. Let $\tau \colon A_{\alpha,k} \to A_{\alpha,k}$ be an analytic map of order $m$ that interchanges the branches of the special fiber. Then after an analytic change of coordinates defined over $k$, $\tau$ is given by either $z \mapsto \zeta_m /z$ or $z \mapsto \zeta_{\frac{m}{2}} a /z$ for some $a$ with $|a| = \alpha$. In particular, $\tau$ has two fixed points and $A_{\alpha,k}/ \langle \tau \rangle$ is a disc.

\end{enumerate}
\end{thm}




\section{Annuli on Superelliptic Curves}

Based on the preceding lemmas, we can write explicit equations for the annuli on the curve. We again follow Stoll's notation. Denote by $\Theta$ the multiset of ramification points of $C$ counted with multiplicity. By a change of coordinates we can always assume that infinity is not a branch point: First move any branch points away from zero, then substitute $$x \mapsto \frac{1}{x} \text{ and } y \mapsto \frac{y}{x^{\left \lceil \frac{\deg{f}}{m} \right \rceil}}.$$ The finite roots of $f$ are inverted, and $f$ picks up a new root of multiplicity $m - [\deg(f) \pmod{m}]$ at zero, where $\deg(f) \pmod{m}$ is the smallest representative between 1 and $m$ of the residue class. We can therefore write $$y^m = f(x) = c\prod_{\theta \in \Theta}(x-\theta).$$
If $\theta \neq 0$, we have the $p$-adic functions $$f^+_\theta(x) = \left(1-\frac{\theta}{x} \right)^{1/m}, f^-_\theta(x) = \left(1-\frac{x}{\theta} \right)^{1/m}.$$ These converge $p$-adically when $|x| > |\theta|$ and $|x| < |\theta|$, respectively. They satisfy the equations
$$x - \theta = xf^+_\theta(x)^m \text{ and } x - \theta = -\theta f^-_\theta(x)^m.$$

The following lemmas are generalizations of those in \cite{stoll2013uniform}. The first gives equations for both discs and annuli whose quotient is a disc. In this case, such a quotient is either completely ramified or completely split.

\begin{lem} \label{discs}
Let $D \subseteq \mathbb{P}^1_k$ be an open disc, with $\phi \colon D_{0,k} \to D$ its parametrization by the open unit disc.

\begin{enumerate}
\item
Suppose $D(\mathbb{C}_p) \cap \Theta = \emptyset$. If there exists $\beta \in D(k)$ such that $f(\beta)$ is not an $m$\ts{th} power in $k$, then $\pi^{-1}(D) \cap C(k)$ is empty. If $f(\beta)$ is an $m$\ts{th} power for all $\beta$, then $\pi^{-1}(D)$ is the disjoint union of $m$ disjoint open discs in $C$, each isomorphic to $D$ via $\pi$.

\item
Suppose $D(\mathbb{C}_p) \cap \Theta = \{ \theta_1\}$, and that $D$ has radius $r$. Then $\theta_1 \in k$. Further assume that there is some $\beta \in k$ such that $r|f'(\theta_1)| = |\beta|^m$. Then $\pi^{-1}(D)$ is a disc on $C$, and up to an analytic change of coordinates the map $\pi$ is just the $m$\ts{th} power map (i.e.\@, the superelliptic automorphism acts by rotation).

\item
Suppose $D(\mathbb{C}_p) \cap \Theta = \{ \theta_1, \theta_2\}$. Then $(x-\theta_1)(x-\theta_2)$ has coefficients in $k$. Furthermore, there are two possibilities for the set $\pi^{-1}(D)$. It is either (a) contained in the preimage of the smallest closed disc containing $\theta_1$ and $\theta_2$, or (b) $\pi^{-1}(D)$ is a union of $\frac{m}{2}$ annuli $A^{(j)}$ in $C$ such that for some $\beta \in k^\times$ we have $\pi(z) = z+ \beta/z$ (after an analytic change of coordinates.) In this case $\tau$ acts as $z \mapsto \beta/z$, $j \mapsto j+1$.

\end{enumerate}
\end{lem}
\begin{proof}
\begin{enumerate}
\item[]
\item
We make a coordinate change on $\mathbb{P}^1$, defined over $k$, so that $D = D_{0,k}$ and $|\theta | \geq 1$ for all $\theta \in \Theta.$ When $\pi^{-1}(D) \cap C(k) \neq \emptyset$, there exists $\gamma \in k^\times$ such that $f(0) = \gamma^m$. Thus on the disc $\pi^{-1}(D)$ the equation for $C$ now takes the form $$y^m = c \prod_{\theta \in \Theta} \theta \left( \prod_{\theta \in \Theta} f^-_\theta \right)^m = \gamma^m h(x)^m.$$ Then $\pi^{-1}(D) = \coprod_i D^{(i)}$, where
$$D^{(i)} \defeq \left \{(z,\zeta_m ^i \gamma h(z)) \mid z \in D\right \}, ~ i=1,\ldots,m.$$

\item
By assumption there exist $\gamma \in k^\times$ and  $u \in \mathcal{O}_k^\times$ such that $\gamma^m = uf'(\theta_1)$. We again change coordinates so that $D = D_{0,k}$ and $\theta_1 = 0.$ On $\pi^{-1}(D)$ the equation for $C$ is given by $$y^m = cx \prod_{0 \neq \theta \in \Theta} (-\theta) \left( \prod_{0 \neq \theta \in \Theta} f^-_\theta \right)^m = c' x h(x)^m. $$ Here $h(x) =  \prod_{0 \neq \theta \in \Theta} f^-_\theta.$ Then we have the parameterization
$$D = \left\{ \left(u z^m, \gamma z h\left(uz^m\right) \right) \mid z \in D_{0,k} \right\}.$$

\item
Changing coordinates again so that $D = D_{0,k}$ and $\theta_1 + \theta_2 = 0$, the equation for the curve is given by $$y^m = c \prod_{\theta \in \Theta \setminus \{\theta_1,\theta_2\}} (-\theta) (x^2 - a) \left( \prod_{\theta \in \Theta \setminus \{\theta_1,\theta_2\}} f^-_\theta \right)^m =  c' (x^2 - a)h(x)^m.$$ Here $h(x) = \prod_{\theta \in \Theta \setminus \{\theta_1,\theta_2\}} f^-_\theta.$ The convergence properties of the $m$\ts{th} root when $(p,m) = 1$ are independent of $m$, and so if $|x| > |a|$ and $c'$ is not an $m$\ts{th} power in $k$, $x$ cannot be the coordinate of a $k$-point. In this case the preimage of $D$ in $C$ is contained in the preimage of $\{ |x| < \theta_1 \}$.

Suppose $c' = \gamma^m$ with $\gamma \in k$. Then we set $$z = \frac{1}{2}\left( x + \left( \frac{y}{\gamma h(x)} \right)^{\frac{m}{2}} \right),$$ and the desired parametrization of $\pi^{-1}(D)$ is given by $$A^{(j)} \defeq \left(z + \frac{a}{4z}, \zeta_{\frac{m}{2}}^j \gamma \left(z - \frac{a}{4z}\right)^{\frac{2}{m}} h\left(z + \frac{a}{4z} \right) \right).$$ Here we've chosen a branch of the $\frac{m}{2}\ts{th}$ root so that $(-1)^\frac{m}{2} = \zeta_m^{-1}$. The map which sends $z \mapsto \frac{a}{4z}$ and $j \mapsto j+1$ fixes the $x$-coordinate and multiplies the $y$-coordinate by $\zeta_m$, so it is the superelliptic automorphism.
\end{enumerate}
\end{proof}

The second lemma describes equations for annuli whose quotient is an annulus, in the process generalizing Stoll's lemma to account for the many possible behaviors of an automorphism of an $m$-to-one mapping. Indeed, while the case $m=2$ separates cleanly into odd and even annuli, we have many more cases.

\begin{lem} \label{annuli}
Let $\phi \colon A_{\alpha,k} \to A \in \mathbb{P}^1_k$ be an open annulus such that $A\cap \Theta = \emptyset$ and $A(k) \neq \emptyset$. The complement of $A$ in $\mathbb{P}^1_k$ is the disjoint union of two closed discs, which partition $\Theta$ into $\Theta_0$ and $\Theta_\infty$. This partition induces a factorization $f(x) = c f_0(x)f_\infty (x)$ with $f_0$ and $f_\infty$ monic such that the roots of $f_0$ are the elements of $\Theta_0$ and the roots of $f_\infty$ are the elements of $\Theta_\infty$. 

Write $d = \gcd(\# \Theta_0, m)$. If $c$ is not a $d$\ts{th} power in $k$, then $\pi^{-1}(A) \cap C(k) = \emptyset$. Assume otherwise, and further suppose that there exists $\gamma \in k$ with $u^dc \left (\prod_{\theta \in \Theta_\infty} (-\theta) \right )= \gamma^m$. Then $\pi^{-1}(A)$ is a union of $d$ annuli, and after an analytic change of coordinates the superelliptic automorphism acts by both interchanging these annuli and rotating them.



\end{lem}


\begin{proof}
By construction $\# \Theta_0$ is invertible modulo $m/d$; let $\# \Theta_0^{-1}$ denote the smallest integer representative for its inverse in this ring. Then $\# \Theta_0 \# \Theta_0^{-1} = 1 + n(m/d)$ for some integer $n$.

Set $h = \left( \prod_{\theta \in \Theta_\infty} f_\theta^+ \right) \left( \prod_{\theta \in \Theta_0} f_\theta^- \right).$ Then the equation for our curve becomes $$y^m = c \left (\prod_{\theta \in \Theta_\infty} (-\theta) \right) x^{\# \Theta_0} h(x)^m.$$ The parameterization is given by $$A^{(j)} = \left \{ \left(u^{d \# \Theta_0^{-1}}z^{m/d}, \zeta_m^{jm/d}\gamma u^n z^{\# \Theta_0 / d}h\left(u^{d \# \Theta_0^{-1}}z^{m/d}\right)\right) \Bigm \vert z \in A_{0,k} \right \}, ~ j=1, \ldots, d.$$ To see the effect of multiplication by $\zeta_m$ on the $y$-coordinate, write $1 = a \frac{\# \Theta_0}{d} + b \frac{m}{d}$ by the Euclidean algorithm. Then we see that the action $z \mapsto \zeta_m^a$, $j \mapsto j+b$ is the same as multiplication by $\zeta_m$ in the $y$-coordinate, and it is thus the superelliptic involution. 
\end{proof}

Combining Raynaud's classification of automorphisms of the annulus with the preceding two lemmata, we conclude the following.

\begin{lem}
The preceding two lemmas provide an exhaustive list of the parameterizations of maximal annuli on the curve $C$.
\end{lem}

\begin{proof}

Suppose $\phi \colon A \to C$ is a maximal annulus such that $A(k) \neq \emptyset$. Then $A$ can be parametrized using Lemmas~\ref{discs} and~\ref{annuli}. We have the following cases.

\begin{enumerate}
\item $\tau(A) \cap A = \emptyset$.

In this case $\pi$ is an isomorphism from $A$ onto its image, and we conclude that $\pi(A)$ is an annulus containing no ramification points. Thus, by Lemma~\ref{annuli}, each component of $\pi^{-1}(\pi(A))$ can be parametrized via the map we described there. In particular, the lemma provides a parameterization of $A$ itself. We must further have $\gcd(m, \Theta_0) > 1$, since $\tau$ sends $A$ to a disjoint annulus. 

\item $\tau(A) \cap A \neq \emptyset$, and $\tau$ preserves the orientation of the chain corresponding to $A$ in the special fiber of $C.$


In this case we have $\tau(A) = A$. The orientation condition combined with Raynaud's classification (Theorem~\ref{raynaud}) implies that (up to an analytic change of coordinates) $\phi^* \tau$ is a rotation of order dividing $m$ on $A_{\alpha,k}$. Thus $\pi(A)$ is an annulus containing no branch points. We conclude that $A$ is parametrized via Lemma~\ref{annuli}. Furthermore, in this case $\gcd(m, \Theta_0) = 1$ since $\pi^{-1}(\pi(A)) = A$.

\item $\tau(A) \cap A \neq \emptyset$, and $\tau$ reverses the orientation of the chain corresponding to $A$.


We claim that this case cannot occur when $m > 2$. As above, $\tau(A) = A$, and Raynaud's classification tells us that $\phi^* \tau$ is an inversion composed with a rotation of order dividing $m$, so $\pi(A)$ is a disc containing two branch points. Then $\pi^{-1}(A)$ is thus parametrized via the proof of Lemma~\ref{discs}, i.e. $\pi^{-1}(A)$ is a disjoint union of $\frac{m}{2}$ annuli, each interchanged by $\tau$. When $m > 2$ (and $m$ is even) this contradicts the assumption that $\tau(A) = A$.
\end{enumerate}
\end{proof}

\section{Bounding Zeros of Differentials}

Our goal in this section is to obtain bounds on the number of zeros of differentials on the annuli which cover our superelliptic curve $C$. In general, the Weierstrass Preparation Theorem \cite[Theorem 2.4.3]{cherry}  says that an analytic function on an annulus $A$ can be written in the form $f = vu$, where $v$ is a Laurent polynomial with finitely many exponents and $u$ has no zeros on $A$. One of Stoll's key insights was that there is a basis for $H^0(C_k,\Omega^1_C)$ in which every basis element has the same $u$ in the Weierstrass decomposition. This uniform description allows us to induce cancellation in the $v$-component.

Differentials on a superelliptic curve with equation $y^m = f(x)$ are much more complicated than those on a hyperelliptic curve, especially when $f(x)$ is not assumed to be separable. (As pointed out by Stoll in email correspondence, we cannot assume separability because when we move the branch point at infinity to zero, say, it may become root of higher multiplicity.) However, Koo in \cite{koo} does describe a number of linearly independent differentials. A consequence of Koo's main theorem is that, for $0\leq i \leq \lfloor \frac{\deg(f)}{m} \rfloor - 2$, the differentials $$\omega^{(i)} = x^i \frac{dx}{y}$$ are holomorphic. (Note that we may have changed $\deg(f)$ by moving ramification away from infinity, but we have only increased it, so we may use $\deg(f)$ in the bound here.) There are, in fact, more available differentials, but their existence depends in a complicated combinatorial way on the factors that their leading polynomials in $x$ share with $f(x)$. At this stage it does not seem particularly useful to count them.

Based on our classification of annuli and discs on $C$, we thus have the following local description of differentials. The only case when $m>2$ is an annulus arising from Lemma~\ref{annuli}, so we have $$\phi^* \omega^{(i)} = z^{(i+1)m/d - \# \Theta_0 /d} \frac{1}{h\left(u^{d \# \Theta_0^{-1}}z^{m/d}\right)} \frac{dz}{z}.$$ We've again used the notation $d = \gcd(m,\# \Theta_0)$.






\begin{thm}
Let $V \neq 0$ be a subspace of codimension at least one of $H^0(C_k,\Omega^1_C)$. Then there exists a nonzero differential $\omega \in V$ such that $\phi^* \omega = v(z) u(z) dz/z$, where $v$ is a finite Laurent series with its highest and lowest exponents differing by at most $m(r+2)/d+1$ and $u$ is an analytic function that is nonzero on $A$.
\end{thm}

\begin{proof}
Let $W$ be the subspace of $H^0(C_k,\Omega^1_C)$ spanned by $\{\omega^{(i)} \mid 0 \leq i \leq r + 2\}$. (This subspace exists by the rank hypothesis $r \leq \lfloor \frac{\deg(f)}{m} \rfloor -4$.) The $\omega^{(i)}$ are linearly independent, and so $W$ has dimension $r + 3$. By assumption $V$ has dimension at least $g-r-2$, so $W \cap V \neq \{0\}$. Let $\omega$ be a nonzero element in this intersection.

As before, $A$ must arise as in Lemma~\ref{annuli}. Then, since $\omega \in W$, we have the description $$\phi^* \omega = \sum_{j = 0}^{r+2} z^{(j+1)m/d - \# \Theta_0 /d} \frac{1}{h\left(u^{d \# \Theta_0^{-1}}z^{m/d}\right)} \frac{dz}{z}.$$ Factoring out $u(z) \frac{dz}{z} \defeq \frac{1}{h\left(u^{d \# \Theta_0^{-1}}z^{m/d}\right)} \frac{dz}{z}$, the remaining sum $v(z)$ has exponents that range from $m/d - \# \Theta_0 /d$ to $m(r+3)/d + m/d - \# \Theta_0 (m-1)/d$. Thus the highest and lowest exponents differ by at most $m(r+2)/d + 1$, inclusive. Furthermore, $v$ has no zeros because $h(x)$ has no poles in the annulus.


\end{proof}

\section{The Final Count(down)}






\subsection{$p$-Adic Rolle's Theorem}

The strategy of effective Chabauty hinges on being able to control the zeros of $\int_P^z \omega$ based on the zeros of $\omega$ using a $p$-adic Rolle's theorem. The theory of Newton polygons (and in more recent applications, tropical geometry) provides the tools for this analysis.

If $p>e+1$, where again $e$ denotes the ramification index of $k$ over $\mathbb{Q}_p$, we define $$\mu \defeq 1 + \frac{e}{p-e-1}.$$

\begin{prop}[\cite{stoll2013uniform}, Prop. 7.7]

Suppose a $p$-adic differential $\omega$ has a Newton polygon of width $w$ in a $p$-adic annulus $A$, and that $p > 2g$. Then $$f_\omega(z) \defeq \int _P^z \omega$$ has at most $\mu w$ zeros on $A$.
\end{prop}

\subsection{Uniform Bounds}
\begin{rem}
Throughout this section we will assume that $K = \mathbb{Q}$ and $k = \mathbb{Q}_p$ for a prime $p$ that we will pick wisely. It is possible to obtain a bound $R(\delta,m,g,r)$ when $K$ is an arbitrary number field of degree $\delta$ in the same way Stoll does, but we emphasize here the explicit nature of the bounds over $\mathbb{Q}$.
\end{rem}

We will need to have a bound for the number of rational points that reduce to smooth points of the special fiber of $C$. This bound has been established in two ways in the literature: Via the alternative rank functions of Katz and Zureick-Brown \cite{katz2015uniform} and through the theory of metrized complexes defined by Amini and Baker \cite{ab}. We restate their theorem here.

For the purposes of this section, it is helpful to denote by $C_D$ the portion of $C(\mathbb{Q}_p)$ covered by discs and by $C_A$ the portion covered by annuli. 

\begin{prop}[\cite{stoll2013uniform}, Lemma 7.1, using \cite{KZB}, Thm. 4.4]

Let $V \neq 0$ be a linear subspace of codimension $r$ of the space of regular differentials on $C$ and let $N_D$ denote the number of discs whose union is $C_D(k)$. Suppose further that $p>e+1$. Then the integrals $f_\omega$ for $\omega \in V$ have at most $$N_D + 2\mu r \leq (5q+2)(g-1) - 3q(g-1) + 2 \mu r.$$ common zeros in $C_D(k)$.

\end{prop}

To calculate the number of rational points lying in $C_A(\mathbb{Q}_p)$, we conclude the discussions of the previous sections. 

\begin{prop}

Suppose $p > e+1$. Then the number of common zeros in $C_A(\mathbb{Q}_p)$ of all $f_\omega$ for $\omega \in V_{\text{chab}}$ is bounded by $$\left( \frac{4g-4}{m} +1 \right)  \mu m(r+3) .$$

\end{prop}

\begin{proof}

Our residue tube analysis shows that, in fact, each orbit of annuli that arise in the case of a non-inverting action of $\tau$ contains at most $m(r+2)+d \leq m(r+3)$ shared zeros. We therefore take $m(r+3)$ as a uniform bound for the number common zeros of the differentials in $V$ on any orbit of annuli. Applying Stoll's Newton polygon calculation for the optimal differential in the residue orbit, this leaves at most $\mu m(r+3)$ common zeros of the integrals $f_\omega$.

We can count the number of orbits of residue annuli. Each orbit corresponds to an edge in the image of a minimal skeleton of $C$ via the analytification of the map $\pi$. This image is obtained by starting with the convex hull of the ramification points in $\mathbb{P}^1$ (a tree with at most $s$ leaves) and removing the leaves. This leaves a tree with at most $s - 2$ nodes, hence at most $s-3$ edges.

Now a computation with Riemann-Hurwitz, worked out by Koo by his paper, shows that $$2g-2 = m(s-1) - \gcd(m, \deg(f)) - \sum_{i=1}^s \gcd(m,n_i).$$ Note that $\gcd(m,\deg(f)) \leq m$ and $\gcd(m,n_j) \leq m/2$, and so $$s \leq \frac{4g-4}{m} + 4.$$ (Again, $\deg(f)$ may have changed through changes of variables, but it doesn't affect this bound.) Thus there are at most $(4g-4)/m + 1$ orbits of annuli. Multiplying this by the number of zeros on each annulus orbit concludes the proof.
\end{proof}
We are now ready to complete the proof of the main theorem modulo the choice of the prime $p$.
\begin{proof}[Proof of Theorem~\ref{main}]
Combining the bounds on $C_D$ and $C_A$, the total number of common zeros of all $f_\omega$ for $\omega \in V_{\text{chab}}$ is bounded by $$\left( (4g-4)/m +1 \right)  \mu m(r+3) + (5p+2)(g-1) - 3p(g-1) + 2 \mu r.$$
\end{proof}

Now we optimize the choice of a prime that satisfies $p>e+1$ and $(p,m) = 1$. The $(p-1)$\ts{th} roots of unity are the only roots of unity in $\mathbb{Q}_p$ for $p$ odd, and so $\zeta_m \in \mathbb{Q}_p$ if and only if $p = 1 \pmod{m}$. For such a prime we have $e=f=1$.

The problem of finding the smallest prime in an arithmetic progression is answered by Linnik's Theorem. The theorem says that there exists an $L$ and $m_0$ such that for all $m>m_0$, the smallest prime congruent to one modulo $m$ is less than a constant times $m^L$. Recent work has shown that $L$ can be taken a little under 5, but at the cost of astronomical bounds for $m_0$. Under the GRH, the smallest prime congruent to one modulo $m$ is less than $m (\log m)^2$. We will content ourselves with an easily digested exponential bound, and trust the reader to search for the smallest prime in an arithmetic progression in any one specific case or use a polynomial bound in general if they want to.

\begin{thm}[\cite{tv}]
The smallest prime congruent to 1 mod $m$ is at most $2^{\phi(m)}-1$.
\end{thm}

Putting all of this together, we conclude that for each $m$, $\# C(\mathbb{Q})$ is bounded by a bilinear polynomial in $r$ and $g$. Furthermore, the dependence on $m$ is at worst polynomial in nature, and can be bounded easily as a function of $2^{\phi(m)}-1$. We also note that one could pick a small prime and analyze its ramification in the field $\mathbb{Q}_p(\zeta_m)$; such a bound would still introduce dependence on $m$.

This concludes the proof of the main theorem.

\begin{rem}

In certain situations, one can use a tower of superelliptic curves to glean arithmetic information about the individual curves. Given the curve $y^m = f(x)$ as above, and the curve $C' \colon y^s = f(x)$ for any divisor $s$ of $m$, we always have a cover $$\rho\colon C \to C'.$$ This cover is given by the map $(x,y) \mapsto (x, y^{m/s})$.

Then a simple argument shows the following.
Suppose that $\# C'(K) \leq B$. Then $\# C(K) \leq R(K) B$, where $R(K)$ denotes the number of $m/s$-th roots of unity in $K$.


We would like to then replace $m$ by its smallest prime divisor and analyze the resulting curve $C'$, but it is not always true that $C'$ satisfies the rank hypothesis of Chabauty's method. It is true, however, that one can relate the genus of the two curves using Riemann-Hurwitz, so if Chabauty's method does work on $C'$ then we may obtain a better bound on the rational points of $C'$.
\end{rem}

\begin{rem}
One might wonder whether the bound in the main theorem can be taken completely independent of $m$. Based on our methods, even the most fanciful conjectures on the least prime in an arithmetic progression still inject some dependence on $m$ into our final bounds. In any case, under RH the dependence is relatively small.
\end{rem}

The main theorem thus provides infinitely many classes of curves $C$ (as we vary $m$ in the main theorem) for which the number of rational points on $C$ is be bounded linearly in $r$ and $g$, and raises the question of whether such a bound might hold for all curves.

\bibliographystyle{plain}
\bibliography{Superelliptic_Paper.bib}


\end{document}